\documentclass[12pt]{amsart}
\usepackage{amsthm,amstext,amssymb,amsmath}

\usepackage{graphicx}

\numberwithin{equation}{section}
\numberwithin{figure}{section}
\theoremstyle{plain}
\newtheorem{thm}{Theorem}
\theoremstyle{plain}
\newtheorem{lem}{Lemma}

 \theoremstyle{remark}
 \newtheorem{rem}[thm]{Remark}
\theoremstyle{definition}
\newtheorem*{defn*}{Definition}
 \theoremstyle{remark}
 \newtheorem*{rem*}{Remark}
 \theoremstyle{plain}
 \newtheorem{cor}[thm]{Corollary}
 \theoremstyle{plain}
 
\theoremstyle{definition}
 
 \theoremstyle{definition}
 \newtheorem{defn}[thm]{Definition}
 \theoremstyle{remark}
 
 \theoremstyle{definition}
 
\theoremstyle{definition}
 \newtheorem{lemma}[thm]{Lemma}
\theoremstyle{definition}
 \newtheorem{assumption}[thm]{Assumption}

\title [BRWs in IID environments]{Survival asymptotics for branching random
walks in IID environments\footnote{Date: \today}}

\author{J\'anos Engl\"ander and Yuval Peres}

\address{Department of Mathematics, University of Colorado
Boulder, Colorado 80309-0395, USA}\email{janos.englander@colorado.edu}
\address{Microsoft Research, One Microsoft Way
Redmond, WA 98052, USA}
\email{peres@microsoft.com}

\begin{document}
\maketitle

\begin{abstract} 
We first study a model, introduced recently in  \cite{ES}, of a critical 
branching random walk in an IID random environment on the $d$-dimensional integer lattice. The walker performs critical (0-2) branching at a lattice point if and only if there is no `obstacle' placed there. The obstacles appear at each site with probability $p\in [0,1)$ independently of each other. We also consider a similar model, where the offspring distribution is subcritical.

Let $S_n$ be the event of survival up to time $n$. We show that on a set  of full $\mathbb P_p$-measure, as $n\to\infty$, 

{\bf (i) Critical case:}
 
\begin{equation*}
P^{\omega}(S_n)\sim\frac{2}{qn};
\end{equation*}

\medskip
{\bf (ii) Subcritical case:}

\begin{equation*}
P^{\omega}(S_n)= \exp\left[\left(
-C_{d,q}\cdot \frac{n}{(\log n)^{2/d}} 
\right)(1+o(1))\right],
\end{equation*}
 where $C_{d,q}>0$ does not depend on the branching law.

Hence, the model exhibits `self-averaging' in the critical case but not in the subcritical one. I.e., in (i) the asymptotic tail behavior is the same as in a `toy model' where space is removed, while in (ii) the spatial survival probability is larger than in the corresponding toy model, suggesting spatial strategies.

We utilize a spine decomposition of the branching process as well as some known results on random walks.
\end{abstract}

\maketitle


\section{Introduction}

\subsection{Model}

We first consider a model, introduced recently in  \cite{ES}, of a critical branching random walk $Z=\{Z_n\}_{n\ge 0}$ in an IID random environment on the $d$-dimensional integer lattice as follows.
The environment is determined by placing \textit{obstacles} on each site, with probability $0\le p<1$, independently of each other.
Given an environment, the initial single particle, located at the origin at $n=0$,
first moves according to a nearest neighbor simple random walk, and
immediately afterwards, the following happens to it (see Fig. 1.1):
\begin{enumerate}
\item If there is no obstacle at the new location (we call it then a {\it vacant site}), 
the particle either vanishes
or splits into two offspring particles, with equal probabilities.
\item If there is an obstacle at the new location, nothing happens to the particle.
\end{enumerate}
The new generation then follows the same rule in the next unit time
interval and produces the third generation, etc.

\begin{figure}[h]
\centering
\centerline{\includegraphics[width=.65\paperwidth]{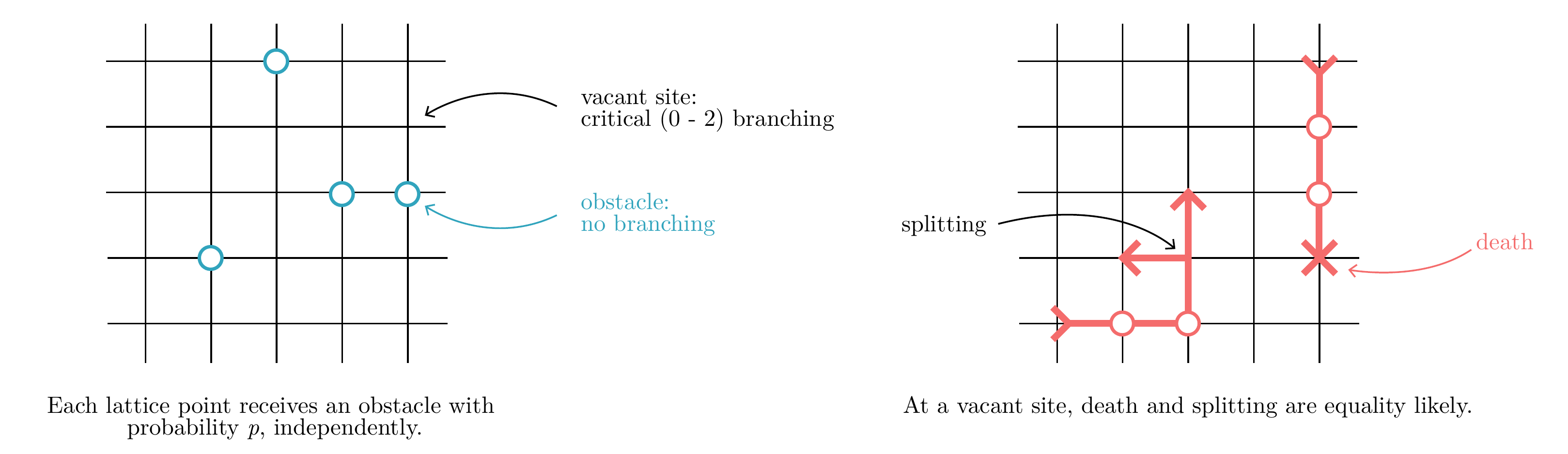}}
\caption{}
\end{figure}

We will also consider the same model when critical branching is replaced by a subcritical one, with mean $\mu<1$. In this latter case we will make the following standard assumption. 

\begin{assumption}[$L\log L$ condition]
In the subcritical case let $L$ denote the random number of offspring, with law $\mathcal{L}$. We assume that  $$\sum_{k=1}^{\infty}{\sf Prob}(L=k)k \log k<\infty.$$
\end{assumption}

Let $p\in[0,1)$. In the sequel, $K=K(\omega)$ will denote the set of lattice points 
with obstacles, $\mathbb{P}_{p}$ will denote the
law of the obstacles and $P^{\omega}$ will denote the law of the BRW
given the environment $\omega\in \Omega$. (Here $\Omega$ may be identified with $\{0,1\}^{\mathbb {Z}^{d}}$.)  Define also $\mathbf{P}_p:=\mathbb{E}_p\otimes{P}^{\omega}$.
We will say that a statement holds `on a set of full $\mathbb{P}_p$-measure,' when it holds under $P^{\omega}$ for $\omega\in\Omega'\subset \Omega$ and $\mathbb{P}_p(\Omega')=1$.

Finally,  $\mathbb{I}_A$ will denote the indicator of the set $A$, and 
for $f,g: (0,\infty)\to (0,\infty)$, the notation $f\sim g$ will mean that $\lim_{t\to\infty}\frac{f(t)}{g(t)}=1$.

\subsection{Quenched survival; main result}

We are interested in the asymptotic behavior, as time tends to infinity,
of the probability that there are surviving particles, and on its possible
dependence on the parameters. 
(Note that in the 
extreme case, when $p=0$, the asymptotic behavior is well known.)

We consider the quenched case, and so, we  can only talk about the \textit{almost
sure} asymptotics, as the probability $P^{\omega}$ itself depends
on the realization of the environment.

Let $S_{n}$ denote the event of survival up to $n\ge0$. That
is, $S_{n}=\{|Z_{n}|\ge1\}$, where $|Z_{n}|$ is the population size
at time $n$. Our  main result will concern the a.s. asymptotic behavior of $P^{\omega}(S_n)$.
\begin{thm}[Quenched  survival probability]\label{main.thm}
Let $d\ge 1$ and $p\in (0,1),$ and recall that $q:=1-p$. Then the following holds on a set  of full $\mathbb P_p$-measure, as $n\to\infty$. 

{\bf (i) Critical case:}
 
\begin{equation}
P^{\omega}(S_n)\sim\frac{2}{qn}\label{quenched.survival.asymptotics};
\end{equation}

\medskip
{\bf (ii) Subcritical case:}

\begin{equation}\label{subcrit.quenched.survival.asymptotics}
P^{\omega}(S_n)= \exp\left[\left(
-C_{d,q}\cdot \frac{n}{(\log n)^{2/d}} 
\right)(1+o(1))\right],
\end{equation}
 where $C_{d,q}$ is a positive constant that does not depend on the branching law.
\end{thm}

\subsection{Motivation; heuristic interpretation}
Consider first the case of critical branching.
Recall  the classic result due to Kolmogorov \cite[Formula 10.8]{H},
that for critical unit time branching with generating function $\varphi$,
as $n\to\infty$, 
\begin{equation}\label{ANK}
P(\text{survival up to}\ n)\sim\frac{2}{n\varphi''(1)}.
\end{equation}

As a particular case, let us consider now a non-spatial toy model
as follows. Suppose that branching occurs with probability $q\in(0,1)$,
and then it is critical binary, that is, consider the generating function
\[
\varphi(z)=(1-q)z+\frac{1}{2}q(1+z^{2}).
\]
 It then follows that, as $n\to\infty$, 
\begin{equation}
P(\text{survival up to}\ n)\sim\frac{2}{qn}.\label{non-spatial}
\end{equation}

Turning back to our spatial model (with critical branching),  simulations suggested (see \cite{ES}) the \textit{self averaging} property of the model:  the asymptotics for the annealed and the quenched case are the same. In fact, this asymptotics is \textit{the same as the one in (\ref{non-spatial})}, where $p=1-q$ is the probability that a site has a obstacle. In other words, despite our model being spatial, in an asymptotic sense, the parameter $q$ simply plays the role of the branching probability of the above non-spatial toy model. To put it yet another way, $q$ only introduces a `time-change.'

In the present paper we would like to establish rigorous results concerning survival.

Our main result will demonstrate that while for critical branching, self-averaging indeed holds, this is not the case for subcritical branching.

For further motivation in mathematics and in mathematical biology, see \cite{ES}. For topics related to the quenched and annealed survival of a single particle among obstacles in a continuous setting, see the fundamental monograph \cite{Sz}. Finally, we mention the excellent current monograph \cite{S15} on branching random walks, which also includes the spine method relevant to this paper.

\medskip
Next, we give a heuristic interpretation of Theorem \ref{main.thm}.

\medskip
\underline{(i) Critical case:}
The intuitive picture behind the asymptotics   is that \textit{there
is nothing the BRW could do to increase
the chance of survival,} at least as far as the leading order term
is concerned (as opposed to well known models, for example when a
single Brownian motion is placed into random medium \cite{Sz}). Hence,
given any environment, the particles move freely and experience branching
at $q$ proportion of the time elapsed, and the asymptotics
agrees with the one obtained in the non-spatial setting as in (\ref{non-spatial}).

Note that whenever the total population size reduces to one, the probability
of that particle staying in the region of obstacles is known to be much less than $\mathcal{O}(1/n)$. So the optimal strategy for this particle to survive is obviously not to try to stay completely in that region and thus avoid branching. Rather, survival will mostly be possible because of the potentially large family tree stemming from that particle.

 Since $Z$ is a $P^{\omega}$-martingale for any $\omega\in\Omega$,
$E^{\omega}(|Z_n|)=1$ for $n\ge 1$, and so 
\begin{equation}\label{reciprocal}
1={E}^{\omega}(|Z_{n}|\mid S_{n})P^{\omega}(S_n).
\end{equation}
In fact, we suspect that
on a set of full $\mathbb{P}_p$-measure, under ${P}^{\omega}(\cdot\mid S_n)$, the law of
$\frac{|Z_{n}|}{n}$ converges to  the exponential distribution with mean $q/2$. (Cf. Theorem C(ii) in \cite{LPP95}.)

\medskip
\underline{(ii) Subcritical case:}
Now the situation is very different. In this case spatial strategy, that is, the avoidance of vacant sites, does make sense, since those sites are now `more lethal.'
Unlike in the critical case, the result now differs from what the non-spatial toy model would give us, namely, in that case, by the previously mentioned result of Heathcote, Seneta and Vere-Jones (Theorem B in \cite{LPP95}),
$$\exists\lim_{n\to\infty}\frac{{\sf Prob}(S_{n})}{\mu^n}>0,$$ under the LlogL condition.
In our spatial setting, the survival probability has thus improved!

Finally, we note that in \cite{ES}, in the annealed case, the second-order survival asymptotics has also been observed through simulations.  Those simulation results\footnote{Simulation has indicated \cite{ES} that $\mathbf{P}_p(S_n)=\frac{2}{qn}+f(n)$, with $n^{2/3}f(n)$ tending to a positive constant.} suggest that spatial survival strategies do exists, which are not detectable at the logarithmic scale but are visible at the second-order level.

\section{Some preliminary results}

In this section we present two simple statements concerning our branching
random walk model which were proven in \cite{ES}, and also some a priori bounds.

\begin{lem}[Monotonicity (Theorem 2.1  in \cite{ES}) and its proof]
\label{monotonicity} Let $0\le p<\widehat{p}\le1$ and fix $n\ge0$.
Then 
\[
{\mathbf{P}}_{p}(S_{n})\le{\mathbf{P}}_{\widehat{p}}(S_{n}).
\]
Also, for any $\omega\in \Omega$ and $n\ge 1$, one has $P^{\omega}(S_n)\ge P^{*}(S_n)$, where $P^{*}$ 
corresponds to the $p=0$ case.
 \end{lem}
Although \cite{ES} only handles the critical case, the same proof carries through for the subcritical case as well. The proof only uses the fact that if $\varphi$ is the generating function of the offspring distribution, then $\varphi (z)\ge z$ on $[0,1]$. This remains the case for subcritical branching too, since $\varphi(1)=1,\varphi '(1)<1$ and $\varphi$ is convex from above on the interval.
\begin{lem}[Extinction (Theorem 2.2  in \cite{ES})]
\emph{}\label{ext} Let $0\le p<1$ and let $A$ denote the event
that the population survives forever. Then, for ${\mathbb{P}}_{p}$-almost
every environment, $P^{\omega}(A)=0.$ \end{lem}
Again, \cite{ES} only handles the critical case, but the same proof carries through for the subcritical case as well. (One then uses that the population size is a supermartingale, instead of a martingale.)

Lemma \ref{monotonicity} yields the following a priori  bounds.
\begin{cor}[A priori  bounds]\label{apriori}  Let
 $f:\mathbb{Z}_+\to (0,\infty)$. Then the following holds.

{\bf (i) Critical case:} On a set of full $\mathbb{P}_{p}$-measure, 
\begin{equation}
    \liminf_{n\to\infty}nP^{\omega}(S_n)\ge 2,\label{elso}
\end{equation}
 and
\begin{equation}
P^{\omega}(|Z_n|>f(n)\mid S_n)=\mathcal{O}\left(\frac{n}{f(n)}\right),\label{masodik}
\end{equation}
as $n\to\infty$.

\medskip
{\bf (ii) Subcritical case:} On a set of full $\mathbb{P}_{p}$-measure, 
\begin{equation}
    \liminf_{n\to\infty}\mu^{-n}P^{\omega}(S_n)>0.\label{elso.subcr}
\end{equation}
\end{cor}
\begin{proof}
(i) In the critical case,
by comparing with the $p=0$ (no blockers) case, when survival is less likely, and for which the non-spatial result 
of Kolmogorov \eqref{ANK} is applicable,
\eqref{elso} follows by Lemma \ref{monotonicity}. Using that $1=E^{\omega}(|Z_n|)=P^{\omega}(S_n)
E^{\omega}(|Z_n|\mid S_n)$, we infer that
$$\limsup_{n\to\infty} \frac{1}{n}E^{\omega}(|Z_n|\mid S_n)\le 1/2.$$ Finally, use the Markov inequality to get \eqref{masodik}.

\medskip
(ii) In the subcritical case ($\mu<1$), the proof is very similar, taking into account  the well known result of Heathcote, Seneta and Vere-Jones (see Theorem B in \cite{LPP95}) that under the $L\log L$ condition, 
\eqref{elso.subcr} holds with limit instead of $\liminf$ for $p=0$.
\end{proof}

\section{Further preparation: Size-biasing and spine in the critical case}

Consider the critical case in this section. Given \eqref{reciprocal}, the asymptotic relation under \eqref{quenched.survival.asymptotics} is tantamount to
\begin{equation}
{E}^{\omega}(|Z_{n}|\mid S_{n})\sim\frac{qn}{2},\label{tantamount}
\end{equation}
as $n\to\infty$. 
We will actually prove that \eqref{tantamount} holds on a set of full $\mathbb P_p$-measure. 

In the particular case when $q=1$ (branching always takes place) and in a non-spatial setting, this has been shown in \cite{LPP95} (see formula (4.1) and its proof on p. 1132).
We will show how to modify the proof in \cite{LPP95} for our case.

(In the subcritical case, we will also reduce the question to the study of the behavior of ${E}^{\omega}(|Z_{n}|\mid S_{n})$ as $n\to\infty$.)

\subsection{Left-right labeling}
At every time of fission, randomly (and independently from everything else in the model) assign `left' or `right' labels to the two offspring. So, from now on, every time we write $P^{\omega}(\cdot\mid S_n)$, we will actually mean,  with a slight abuse of notation,  $P^{\omega}(\cdot\mid S_n)$, augmented with the choice of the labels; we will handle $P^{\omega}(\cdot\cap S_n)$ similarly.  
Ignoring space, and looking only at the genealogical tree, we say that at time $n$, a particle is `to the left' of another one, if, tracing them back to their most recent common ancestor, the first particle is the descendant of the left particle right after the fission. Transitivity is easy to check and thus a total ordering of particles at time $n$ is induced.

\subsection{The size-biased critical branching random walk}
Recall that if $\{p_k\}_{k\ge 0}$ is a probability distribution on the nonnegative integers with  expectation $m\in (0,\infty)$, then the corresponding {\it size-biased distribution} is defined by $\widehat{p}_k:=kp_k/m$ for $k\ge 1$. We will denote the size biased law obtained from $\mathcal{L}$ by $\widehat{\mathcal{L}}$.

Given the environment $\omega$,  the {\it size-biased critical branching random walk} with corresponding law $\widehat{P}^{\omega}$ is as follows. 
\begin{itemize}
    \item The initial particle does not branch until the first time it steps on a vacant site, at which moment it splits into a random number offspring according to $\widehat{\mathcal{L}}$.  
    \item One of the  offspring is picked uniformly (independently from everything else) to be designated as the `spine offspring.' The other offspring launch  copies of the original branching random walk (with $\omega$ being translated according to the position of the site). 
    \item Whenever the `spine offspring' is situated next time at a vacant site, it splits into  a random number offspring according to $\widehat{\mathcal{L}}$, and the above procedure is repeated, etc. 
\end{itemize}
\begin{defn}[Spine]
The distinguished line of decent formed by the successive spine offspring will be called the {\it spine.}
\end{defn}
Note the following.
\begin{itemize}

\item[(i)] {\bf (Survival)} Because of the size biasing, the new process is immortal $\widehat{P}^{\omega}$-a.s.

\item[(ii)] {\bf (Martingale change of measure)} For any given $\omega$, the law of the size-biased critical random walk satisfies that $$\left.\frac{\mathrm{d}\widehat{P}^{\omega}}{\mathrm{d}P^{\omega}}\right\vert_{\mathcal{F}_{n}}=|Z_n|,$$
where $\{\mathcal{F}_{n}\}_{n\ge 0}$ is the natural filtration of the branching random walk, and the lefthand side is a Radon-Nikodym derivative on $\mathcal{F}_{n}$.
This is a change of measure by the nonnegative, unit mean martingale $|Z|$. The proof is essentially the same as in \cite{LPP95}. Even though in that paper the setting is non-spatial, it is easy to check that the proof carries through in our case, because the mean offspring number is always one, irrespective of the site. (See p.1128 in \cite{LPP95}.)
\end{itemize}

In particular,  when the critical law $\mathcal{L}$ is binary (either $0$ or $2$ offspring, with equal probabilities), the law $\widehat{\mathcal{L}}$ is deterministic, namely it is dyadic (that is, $2$ offspring with probability one). In this case, the spine particle always splits into two at vacant sites.

In addition to $\widehat{P}^{\omega}$, we also define the law $\widehat{P}^{\omega}_{*}$ which is the distribution of the size-biased branching random walk, {\it augmented with} the designation of the spine within it. The corresponding, augmented filtration, $\{\mathcal{G}_n\}_{n\ge 0}$ is richer than $\{\mathcal{F}_{n}\}_{n\ge 0}$, as it now keeps track of the position of the spine as well.

The significance of the new law $\widehat{P}^{\omega}$ is as follows. 
Let us denote the spine's path up to $n$ by  $\{X_i\}_{0\le i\le n}$. 
Let 
$$A_n:=\{\text{The spine particle is the leftmost particle of}\ Z_n.\}.$$ 
Then, size biasing and conditioning on $A_n$ has the combined effect of simply conditioning the process on survival up to $n$. That is, the distribution of $Z$ restricted on $\{\mathcal{F}_{n}\}$ is the same
under
${P}^{\omega}(\cdot\mid S_{n})$ and under $\widehat{P}^{\omega}_{*}(\cdot\mid A_{n}).$
To see why this is true, let $C_{n,k}:=\{|Z_n|=k\}$. 
One has for $F\in \mathcal{F}_{n}$ that
\begin{align*}
     \widehat{P}^{\omega}_{*}(F\mid A_{n})&=\frac{\widehat{P}^{\omega}_{*}(F\cap A_{n})}{\widehat{P}^{\omega}_{*}( A_{n})}
     =\frac{\sum_{k\ge 1}\widehat{P}^{\omega}_{*}(F\cap A_{n}\cap C_{n,k})}{\sum_{k\ge 1}\widehat{P}^{\omega}_{*}( A_{n}\cap C_{n,k})}\\ &=\frac{\sum_{k\ge 1}(1/k)\widehat{P}^{\omega}(F\cap C_{n,k})}{\sum_{k\ge 1}(1/k)\widehat{P}^{\omega}(C_{n,k})}\\
    &=\frac{\sum_{k\ge 1}(1/k)E^{\omega}(|Z_n|;F\cap C_{n,k})}{\sum_{k\ge 1}(1/k)E^{\omega}(|Z_n|; C_{n,k})}
    =\frac{\sum_{k\ge 1}P^{\omega}(F\cap C_{n,k})}{\sum_{k\ge 1}P^{\omega}(C_{n,k})}\\
    &=\frac{P^{\omega}(F\cap S_n)}{P^{\omega}(S_n)}=P^{\omega}(F\mid S_n).
    \end{align*}
In particular, 
\begin{equation}\label{turn.to.size-biased}
{E}^{\omega}(|Z_{n}|\mid S_{n})=\widehat{E}_{*}^{\omega}(|Z_{n}|\mid A_{n}).
\end{equation}

\subsection{Frequency of vacant sites along the spine in the critical case}\label{frequency}
Let the branching be critical, and 
let $$L_n:=\sum_{i=1}^n \mathbb{I}_{\{X_{i}\in K^c\}}$$ denote the (random) amount of  time spent by $X$ (the spine) on vacant sites between times $1$ and $n$. 
\begin{lemma}[Frequency of visiting vacant sites]\label{qn} On a set of full $\mathbb P_p$-measure,  
$$\lim_{n\to\infty} \widehat{P}_{*}^{\omega}\left (\left|\frac{L_n}{n}-q\right|>\epsilon \right)=0,\ \forall \epsilon>0.$$
\end{lemma}
\begin{proof}

Let $$F_{n,\epsilon}:=\bigcup_{1\le i\le |Z_n|}\left\{\left|\frac{L^i_n}{n}-q\right|>\epsilon\right\}\in\mathcal{F}_n,$$ where
${L^i_n}$ is defined similarly to $L_n$ for $Z^{i,n}$, the $i^{th}$ particle in $Z_n$ on $S_n$; we define $F_{n,\epsilon}:=\emptyset$ on $S_n^c$. Since $A_n$ concerns only labeling and is independent of the event $\left \{\left|\frac{L_n}{n}-q\right|>\epsilon \right\}$, one has
$$ 
\widehat{P}^{\omega}_{*}\left (\left|\frac{L_n}{n}-q\right|>\epsilon \right)=
\widehat{P}^{\omega}_{*}\left (\left|\frac{L_n}{n}-q\right|>\epsilon \mid A_n\right)
\le
\widehat{P}_{*}^{\omega}\left (F_{n,\epsilon}\mid A_n\right),$$
and switching back to the original measure now, the righthand side equals
\begin{equation*}
{P}^{\omega}\left (F_{n,\epsilon}\mid S_n\right)=
\frac{
P^{\omega}
\left (F_{n,\epsilon}\cap\ S_n
\right)}{P^{\omega} (S_n)}\le \frac{
P^{\omega}
\left (F_{n,\epsilon}
\right)}{P^{\omega} (S_n)}.
\end{equation*}
In view of Corollary  \ref{apriori}(i), it is  sufficient to show that on a set of full $\mathbb P_p$-measure,
$
P^{\omega}\left (F_{n,\epsilon}\right)=o(1/n),
$
as $n\to\infty$.
 
To this end, let $f$ be a positive function on the positive integers, to be specified later. Using  the union bound,
$$P^{\omega}\left(F_{n,\epsilon}\right)\le f(n) Q^{\omega}\left(\left|\frac{T_{n}}{n}-q\right|>\epsilon\right)+P^{\omega}(|Z_n|>f(n)\mid S_n),$$
where $T_n$ denotes the time spent on vacant sites between  times $1$ and $n$ by a simple random walk in the environment $\omega$, starting at the origin  with corresponding probability $Q^{\omega}$.
We will use   $\mathbb{Q}_p$ for the law of the environment.
 
By Corollary  \ref{apriori}(i), on a set of full $\mathbb P_{p}$-measure, the second term on the righthand side is $\mathcal{O}\left(\frac{n}{f(n)}\right)$ as $n\to\infty$. Therefore it is enough to find a function $f$ such that 
\begin{equation}\label{bigger.than.n^2}
n^2=o(f(n))
\end{equation}
and that on a set of full $\mathbb Q_{p}$-measure,
\begin{equation}\label{smaller.than.1/n}
 f(n) Q^{\omega}\left(\left|\frac{T_{n}}{n}-q\right|>\epsilon\right)=o(1/n)\ \text{as} \ n\to\infty.
\end{equation}
 
Observe that  it is sufficient to verify the upper tail large deviations. Indeed,
the lower tail large deviations for the time spent in $K^c$ can be handled similarly, since they are exactly the upper tail large deviations for the time spent in $K$.
 
The statement reduces to one about a $d$-dimensional  {\it random walk in random scenery}.\footnote{Random walks in random scenery (RWRS) were first introduced, in dimension one, by Kesten-Spitzer and also by  Borodin, in 1979.  (See e.g. \cite{GPP13,GPdS14}.)} We now have to consider a scenery  such that it assigns the value 1 to each lattice point with probability $q$, and the value $0$ (that is, the scenery is the indicator of vacancy). With a slight abuse of notation, we will still use  $Q^{\omega}$ and $\mathbb{Q}_p$ for the corresponding laws.

Since it was easier to locate the corresponding annealed result in the literature (see also the remark at the end of the proof), we will use that one, and then show how one easily gets the quenched statement from the annealed one.


To this end, define the random variable $Y^*:=Y-q$, where $Y$  is the `scenery variable,' that is, $Y=1$ with probability $q$ and $Y=0$ otherwise. Then $Y^*$ is centered, and defining $$T_n^*:=\sum_{k=1}^{n}Y^*(X_k),$$ one can apply Theorem 1.3 in \cite{GKS}, yielding that for $\epsilon>0$, 
\begin{align*}
\lim_{n\to\infty}n^{-\frac{d}{d+2}}\log (E_{\mathbb Q_{p}}\otimes Q^{\omega})\left(\frac{T_n}{n}>q+\epsilon\right)&=\\
\lim_{n\to\infty}n^{-\frac{d}{d+2}}\log (E_{\mathbb Q_{p}}\otimes Q^{\omega})\left(\frac{T_n^*}{n}>\epsilon\right)&=-C_{\epsilon}<0.
\end{align*}
We now easily obtain the quenched result too, since for any positive sequence $\{a_n\}_{n\ge 0}$, the Markov inequality yields
$$
\mathbb Q_p\left(Q^{\omega}\left(\frac{T_n}{n}>q+\epsilon\right)>a_n\right)\le (a_n)^{-1}(E_{\mathbb Q_{p}}\otimes Q^{\omega})\left(\frac{T_n}{n}>q+\epsilon\right).$$
Given that $(E_{\mathbb Q_{p}}\otimes Q^{\omega})\left(\frac{T_n}{n}>q+\epsilon\right)=\exp\left(-C_{\epsilon}n^{\frac{d}{d+2}}(1+o(1))\right)$, to finish the proof, we can pick any sequence  satisfying that
$$\sum_n a_n^{-1}\exp\left(-\frac{1}{2}C_{\epsilon}n^{\frac{d}{d+2}}\right)<\infty.$$ Then, by the Borel-Cantelli Lemma,  $$\mathbb Q_p\left(Q^{\omega}\left(\frac{T_n}{n}>q+\epsilon\right)>a_n\ \text{occurs finitely often}\right)=1.$$
Clearly, if we pick, for example,  $a_n:=1/n^4$, then  \eqref{bigger.than.n^2}
and \eqref{smaller.than.1/n} are satisfied (take $f(n):=n^{5/2}$).
\end{proof}
\begin{rem}
Regarding RWRS, we note that in  Theorem 2.3 in \cite{AC03} the {\it quenched} large deviations have been studied in  a more continuous version, namely for a Brownian motion in a random scenery, where the scenery is constant on blocks in $\mathbb R^d$. 
\end{rem}

\section{Proof of Theorem \ref{main.thm} --- critical case}
Our goal is to verify \eqref{tantamount}. To this end, note
that the spine has a (nonnegative) number of `left bushes'  and a (nonnegative) number of `right bushes' attached to it; 
each such bush is a branching tree itself. 
A `left bush' (`right bush') is formed by particles which are to the left (right) of the spine particle at time $n$.
It is clear that, under conditioning on $A_n$, each left bush dies out completely by time $n$  (see Fig. 5.1).
\begin{figure}[h]
	\centering
	\centerline{\includegraphics[width=.85\textwidth]{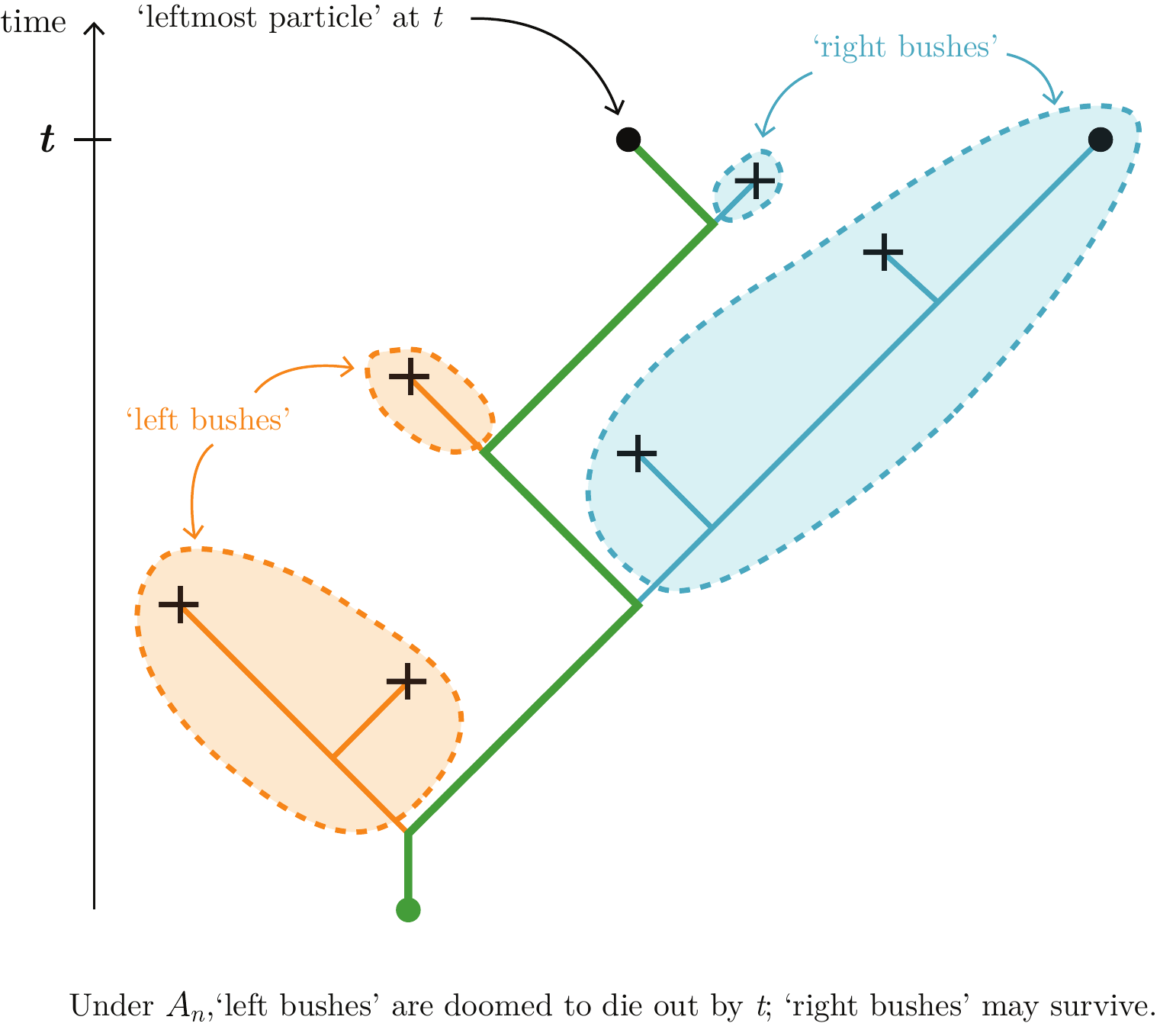}}
    \caption{}
\end{figure}

Because of \eqref{turn.to.size-biased}, we are left with the task of showing that
\begin{equation}
\widehat{E}_{*}^{\omega}(|Z_{n}|\mid A_{n})\sim\frac{qn}{2},\ n\to\infty.\label{tantamount.but.size.biased}
\end{equation}
The proof of this statement is similar to the proof of (4.1) in \cite{LPP95} (with $\sigma^2=1$), except that, as we will see, now we  also have to show that
\begin{equation}\label{last.statement}
\widehat{E}_{*}^{\omega}(\mathrm{number\ of\ all\ bushes\ along\  the \ spine})= qn(1+o(1)).
\end{equation}
The reason is that in \cite{LPP95}, the spine particle branched at every unit time, which is not the case now. In our case,  the spine $\{X_i\}_{1\le n}$ splits into two at each vacant site and thus bushes are attached each time (larger than zero and smaller than $n$) when $X$ is at a vacant site. 

For $n\ge 1$ given, define the set of indices  $$J_n:=\{1\le j\le n-1\mid X_j\in \mathbb{Z}^d\ \text{is a vacant site}\};$$ then \eqref{last.statement} can be written as
$
\widehat{E}_{*}^{\omega}(|J_n|)= qn(1+o(1))
$ (cf. equation \eqref{as.promised} in the sequel).
Furthermore, let $j\in J_n$ and
\begin{itemize}
\item $(LB)_j$ be the event that there is a left bush launched from the space-time point $(X_j,j)$; 

\item $(RB)_j$ be the event that there is a right bush launched from $(X_j,j)$;

\item $(LBE)_j$ be the event that there is a left bush launched from $X_j$ which becomes extinct by $n$.

\item $A_{n,j}:=(RB)_j\cup (LBE)_j$.

\end{itemize}
Then  \begin{equation}\label{cap}
A_n=\bigcap_{j\in J_n} A_{n,j},
\end{equation}
 where the events in the intersection are independent under $\widehat{P}^{\omega}$. (See again Fig. 5.1.)
Conditioning on $A_n$ can be obtained by conditioning successively on $A_{n,j},\ j\in J$.

For $j\in J_n$, let the random variable $R_{n,j}$ be the `right-contribution' of the $j^{th}$ bush to $|Z_n|$. That is, $R_{n,j}=0$ on $(LB)_j,$ and on $(RB)_j$ it is the contribution of the right bush, stemming  from $(X_j,j)$, to $|Z_n|$. The `left contribution' $S_{n,j}$ is defined similarly, and $Z_{n,j}:=R_{n,j}+S_{n,j}$ is the total contribution. Note that $A_{n,j}=\{S_{n,j}=0\}.$

Let $\{R'_{n,j}\}_{j\in J_n}$  be independent random variables under a law $\widetilde Q^{\omega}$ such that 
$$\widetilde Q^{\omega}(R'_{n,j}\in \cdot)=\widehat{P}_*^{\omega}(R_{n,j}\in\cdot\mid A_{n,j}),$$ and let  ${\sf Q}^{\omega}:=\widehat{P}^{\omega}_*\times\widetilde Q^{\omega}$, with expectation $E_{\sf Q}^{\omega}$. Furthermore, let  $R^*_{n,j}:= \mathbb{I}_{A_{n,j}}R_{n,j}+\mathbb{I}_{A^c_{n,j}}R'_{n,j},$ and  $R^*_{n}:=\sum_{j\in J} R^*_{n,j}.$ Then, for $j\in J_n$,
\begin{align}\label{turning.to.Q}
\widehat{P}_{*}^{\omega}(Z_{n,j}\in\cdot\mid A_{n,j})&=\widehat{P}_{*}^{\omega}(R_{n,j}+S_{n,j}\in\cdot\mid A_{n,j})
={\sf Q}^{\omega}(R^*_{n,j}\in\cdot).
\end{align}
(The $S_{n,j}$ term in the second probability has zero contribution.)
Using   \eqref{cap} along with \eqref{turning.to.Q}, it follows that 
\begin{align*}
&\frac{1}{n}\widehat{E}_{*}^{\omega}\left(|Z_{n}|\mid A_{n}\right)=\frac{1}{n}\widehat{E}_{*}^{\omega}\left(|Z_{n}|\mid \bigcap_{j\in J} A_{n,j}\right)=\\
&    \frac{1}{n}\widehat{E}_{*}^{\omega}\left(1+\sum_{j\in J_{n}}Z_{n,j}\mid \bigcap_{j\in J} A_{n,j}\right) =\\
& E_{\sf Q}^{\omega}\left(\frac{1}{n}+\frac{1}{n}\sum_{j\in J_{n}} R^*_{n,j}\right)=
\frac{1}{n}+E_{\sf Q}^{\omega}\left(\frac{R_n^*}{n}\right).
\end{align*}
Hence, the desired assertion \eqref{tantamount.but.size.biased} will follow once we show that (on a set of full $\mathbb{P}_p$-measure) $$\lim_{n\to\infty}E_{\sf Q}^{\omega}\left(\frac{R_n^*}{n}\right)=\frac{q}{2}.$$ 
Denoting $R_{n}:=\sum_{j\in J} R_{n,j},$ the same proof as in \cite{LPP95} reveals that
\begin{equation}\label{without.proof}
\lim_{n\to\infty}E_{\sf Q}^{\omega}\,\frac1n |R_n-R_n^*|= 0.
\end{equation}
(The intuitive reason is that $A^c_{n,j}=\{S_{n,j}>0\}$, while the probability of the survival of a bush 
tends to zero as the height of the bush tends to infinity; thus $A^c_{n,j}$ only occurs rarely.
The fact that now we do not have a bush launched at every position of the spine makes the estimated term even smaller.)

In view of \eqref{without.proof}, it is sufficient to show that $\lim_{n\to\infty}\widehat{E}^{\omega}_{*} \left(R_{n}/n\right)=q/2$.
Since the $\{R_{n,j}\}_{j\in J_{n}}$ are independent of $|J_n|$ and $\widehat{E}^{\omega}_{*} R_{n,j}=\frac12$ for each $j\in J_n$ 
(as each bush is equally likely to be left or right under $\widehat{P}^{\omega}_* $), one has
$$
\widehat{E}^{\omega}_{*}\left(R_{n}/n\right)=\widehat{E}^{\omega}_{*} \left(\frac{1}{n}\sum_{j\in J_{n}} R_{n,j}\right)
=\frac12\widehat{E}_{*}^{\omega}(|J_n|/n)=\frac12\widehat{E}_{*}^{\omega}(|L_{n-1}|/n),
$$
where we are using the notation of Lemma \ref{qn} ($|J_{n+1}|=L_{n}$).
Hence, our goal is to show that
\begin{equation}\label{as.promised}
\lim_{n\to\infty}\widehat{E}_{*}^{\omega}(|L_n|/n)=q.
\end{equation}
Write
\begin{eqnarray*}
\widehat{E}_{*}^{\omega}(L_n)=&&\!\!\!\!\!\widehat{E}_{*}^{\omega}\left(L_n\mid \frac{L_{n}}{n}\in (q-\epsilon,q+\epsilon)\right)
\widehat{P}_{*}^{\omega}\left(\frac{L_{n}}{n}\in (q-\epsilon,q+\epsilon)\right)
\\  +&&\!\!\!\!\!\widehat{E}_{*}^{\omega}\left(L_n\mid \frac{L_{n}}{n}\not\in (q-\epsilon,q+\epsilon)\right)\widehat{P}_{*}^{\omega}
\left(\frac{L_{n}}{n}\not\in (q-\epsilon,q+\epsilon)\right).
\end{eqnarray*}
Now use Lemma \ref{qn}.
Since the first probability on the righthand side is $1-o(1)$ and the second is $o(1)$, and since $0\le L_n\le n$, we have that
 $$ n(q-\epsilon)(1-o(1))\le \widehat{E}_{*}^{\omega}(L_n)\le (q+\epsilon)n+ o(n),$$
that is,
$$ q-\epsilon-o(1)
\le \widehat{E}_{*}^{\omega}(L_n)/n\le 
q+\epsilon+ o(1).$$
Since $\epsilon$ is arbitrary, we are done. $\qed$

\section{Proof of Theorem \ref{main.thm} --- subcritical case}
Recall the definition of $Q^{\omega}=Q_q^{\omega}$  from the proof of Lemma \ref{qn} and that $K$ is the `total obstacle configuration.' Just like in Subsection \ref{frequency}, let  $T_n$ be the time spent in $K^c$ (vacant sites).

Let $Y$ be a simple random walk on $\mathbb Z^d$ with `soft killing/obstacles' under $Q_q^{\omega}$  and let $\mathcal{E}_{q}^{\omega}$ denote the corresponding expectation. By `soft killing' we  mean that at each vacant site, independently, the particle is killed with probability $1-\mu$.
Let\footnote{As the reader has probably guessed already, $\mathtt{DV}$ refers to `Donsker-Varadhan.'} $$\mathtt{DV}_q(n)=\mathtt{DV}(\omega,q,n,\mu):=\mathcal{E}_{q}^{\omega}(\mu^{\sum_{1}^{n}1_{K^{c}}(Y_{i})})=\mathcal{E}_{q}^{\omega}(\mu^{T_n})$$ be the quenched probability that $Y$  survives up to time $n$.
Hence $q$ plays the role of the `intensity' and $\mu$ plays the role of the `shape function' in this discrete setting.

It is known  that for almost every $\omega$,
\begin{equation}\label{Antal.DV}
\mathtt{DV}_q(n)=\exp\left[-C_{d,q}\cdot \frac{n}{(\log n)^{2/d}} (1+o(1))\right],
\end{equation}
as $n\to\infty$,
and $C_{d,q}>0$ does not depend on $\mu$. 
See formula (0.1) on p.58 of \cite{AntalThesis} for hard obstacles. In fact, the proof  for hard obstacles extends  for soft obstacles. Indeed, it becomes easier, since in the case of soft obstcles one does not have to worry about `percolation effects,' that is that the starting point of the process is perhaps not in an infinite trap-free region. Clearly, the lower estimate for survival among hard obstacles is still valid for soft obstacles; the method of proving the upper estimate is a discretized version of Sznitman's `enlargement of obstacles' in both cases.
(See also \cite{Antal.AOP} for similar results and for the enlargement technique in the discrete setting.)


Returning to our branching process and the event $S_n$, we first show that on a set of full $\mathbb{P}_p$-measure, 
\begin{equation}\label{subcrit.reduces.to.expectation.too}
P^{\omega}(S_n)=\frac{\mathtt{DV}_q(n)}{E^{\omega}(|Z_n|\mid S_n)}.
\end{equation}
The expectation $E^{\omega}|Z_n|$ can in fact be expressed as a functional of a {\it single} particle (this is the `Many-To-One' Lemma for branching random walk): $$E^{\omega}|Z_n|=\mathcal{E}_q^{\omega}\,\left(\mu^{T_n}\right)={\mathtt{DV}_q(n)}.$$
This follows from the fact that for $u_n(x)=u^{\omega}_n(x):=E_x^{\omega}|Z_n|$, one has the recursion $$u_n(x)=\sum_{y\sim x} u_{n-1}(y)(\mathbb{I}_{\{y\in K\}}+\mu\, \mathbb{I}_{\{y\in K^c\}})p(x,y),$$
where $y\sim x$ means that $y$ is a neighbor of $x$, and $p(\cdot,\cdot)$ is the one-step kernel for the walk.
(See also \cite{S15}.) Thus
\begin{equation}
    P^{\omega}(S_n)E^{\omega}(|Z_n|\mid S_n)=E^{\omega}|Z_n|=
    \mathcal{E}_q^{\omega}\,\mu^{T_n},
\end{equation}
proving \eqref{subcrit.reduces.to.expectation.too}. 
Since the denominator on the righthand side of \eqref{subcrit.reduces.to.expectation.too} is at least one,  it follows that
\begin{equation}\label{asbefore}
P^{\omega}(S_n)\le \mathtt{DV}^{\mu}_q(n),
\end{equation}
where we emphasize the dependence on $\mu$.

On the other hand, we claim that 
\begin{equation}\label{mustar}
P^{\omega}(S_n)\ge \mathtt{DV}^{\mu^*}_q(n),
\end{equation}
where $\mu^*:=1-p_0$, and $p_0>0$ is the probability of having zero offspring (i.e. death) for the law $\mathcal{L}$.

The reason \eqref{mustar} is true is the following coupling argument:
Every time the particle is at a vacant site, with probability $1-p_0$, we can pick uniformly randomly (independently from everything else) a `successor,' thus effectively embed a random walk with killing into the branching random walk. Clearly, the probability that we are able to complete this procedure until time $n$ is exactly the probability that a single random walk with soft killing survives, where soft killing means that it is killed with probability $p_0$, independently at each vacant site.

It is also evident that if we are able to complete this procedure until time $n$ then the branching process has survived up to time $n$.

Having \eqref{asbefore} and \eqref{mustar} at our disposal, we can now conclude  the assertion of Theorem \ref{main.thm}(ii), because $\mathtt{DV}^{\mu}_q(n)$ and $\mathtt{DV}^{\mu^*}_q(n)$ both have the asymptotic behavior given in \eqref{Antal.DV}, despite the fact that $\mu>\mu^*$ in general.
$\qed$

\medskip
{\bf Acknowledgments:} The hospitality of Microsoft Research, Redmond and of the University of Washington, Seattle are gratefully acknowledged by J.E.

\end{document}